\documentclass{amsart}

\usepackage{amssymb}
\usepackage{amsfonts}
\usepackage{amsthm}

\theoremstyle{plain}
\newtheorem{theorem}{Theorem}[section]
\newtheorem{corollary}[theorem]{Corollary}
\newtheorem{proposition}[theorem]{Proposition}

\theoremstyle{definition}

\newtheorem{remark}[theorem]{Remark}
\newtheorem{example}[theorem]{Example}

\numberwithin{equation}{section}

\newcommand{\R}{{\mathbb R}}

\begin{document}

\title[Boundedness of Positive Integral Operators on Lorentz-Gamma Spaces]%
{Boundedness of Positive Integral Operators on Lorentz-Gamma Spaces}

\author[R.~Kerman]{Ron Kerman}
\address{Department of Mathematics and Statistics,
Brock University, 1812 Sir Isaac Brock Way, St.~Catharines, ON L2S~3A1, Canada}
\email{rkerman@brocku.ca}

\author[S.~Spektor]{S. Spektor}
\address{Quantitative Science Department,
Canisius University, 2001 Main Street, Buffalo, NY 14208-1098, USA}
\email{spektors@canisius.edu}

\begin{abstract}
We characterize the boundedness of a positive integral operator $T_K$,
with kernel $K\in M_+(\R^{2n})$, between Lorentz-Gamma spaces
$\Gamma_{p,\phi_2}(\R^n)$ and $\Gamma_{q,\phi_1}(\R^n)$, $1<p\le q<\infty$.
The key step reduces the $n$-dimensional problem to a one-dimensional
weighted norm inequality for the composed operator $T_LS$, where
$L=(K^{*_2})^{*_1}$ is the iterated rearrangement of $K$ introduced by
Blozinski~\cite{B} and $S$ is the Stieltjes transform.  Explicit
Muckenhoupt-type conditions are obtained for the case $L(t,s)=(t+s)^{-1}$,
corresponding to the iterated Stieltjes operator $S^2$.
\end{abstract}

\maketitle

\section{Introduction}

The Lorentz-Gamma (LG) spaces $\Gamma_{p,\phi}$, defined by the norm
\begin{equation}\label{LGnorm}
\rho_{p,\phi}(f)=\left[\int_{\R_+}f^{**}(t)^p\phi(t)\,dt\right]^{1/p},
\quad f^{**}(t)=t^{-1}\int_0^tf^*(s)\,ds,
\end{equation}
were systematically studied by Gogatishvili and Kerman~\cite{GK}, who
characterized their r.i.\ properties and established when $\rho_{p,\phi}$
is a genuine norm.  Here $f^*$ denotes the nonincreasing rearrangement of
$f$ and the weight $\phi\in M_+(\R_+)$ satisfies
\begin{equation}\label{phi-cond}
\int_{\R_+}\frac{\phi(t)}{1+t^p}\,dt<\infty
\quad\text{and}\quad
\int_{\R_+}\phi(t)\,dt=\infty.
\end{equation}
The class $\Gamma_{p,\phi}$ contains the classical Lorentz spaces
$L_{p,q}$ (take $\phi(t)=t^{q/p-1}$, $1\le q<\infty$) and the Lebesgue
spaces $L^p$ (take $\phi\equiv 1$) as special cases; see
\cite[Ch.\,4]{BS}.  The K\"othe dual of $\rho_{p,\phi}$ is equivalent to
a norm of the same LG type, with an explicit dual weight computed in
\cite{GolK,GP}.

Given a nonneg\-ative kernel $K\in M_+(\R^{2n})$, we consider the
boundedness of the integral operator
\[
(T_Kf)(x)=\int_{\R^n}K(x,y)f(y)\,dy
\]
from $\Gamma_{p,\phi_2}(\R^n)$ to $\Gamma_{q,\phi_1}(\R^n)$,
$1<p\le q<\infty$; that is, the existence of $C>0$ such that
\begin{align}
\rho_{q,\phi_1}(T_Kf)\le C\,\rho_{p,\phi_2}(f), \quad
0\le f\in \Gamma_{p,\phi_2}(\R^n).
\end{align}
Norm inequalities of this type have a rich history. For the classical
Lebesgue case, $\phi_i\equiv 1$, $p=q$, characterizations go back to
Schur's test and the Muckenhoupt $A_p$ theory; see \cite{BS}.  Weighted
Hardy-type operators, a special subclass of $T_K$, were characterized in
Lorentz and Orlicz settings by Bloom and Kerman~\cite{BK}, whose
Muckenhoupt-type conditions we invoke in Section~\ref{sec:example}.
The present work extends to the LG setting the approach introduced in
\cite{KS1} for Orlicz-Lorentz norms; a detailed historical account of
the boundedness problem for $T_K$ can be found in Section~4 of that paper.

The strategy is to replace $T_K$ by a one-dimensional operator with a
more tractable kernel.  The \emph{iterated rearrangement}
$L=(K^{*_2})^{*_1}$ of $K$, introduced by Blozinski~\cite{B} in the
context of mixed-norm spaces, is defined as follows: for each $x\in\R^n$
rearrange $k_x(y)=K(x,y)$ in $y$ to get $K^{*_2}(x,s)$, then rearrange
in $x$ to arrive at $L(t,s)=(K^{*_2})^{*_1}(t,s)$, $s,t\in\R_+$. By
construction, $L(t,s)$ is nonincreasing in each of $s$ and $t$.

Arguing as in Theorem~2.5 of \cite{KS1} one obtains
\[
\int_0^t(T_Kf)^*(s)\,ds\le\int_0^t(T_Lf^*)(s)\,ds,\quad t\in\R_+,
\]
and the Hardy-Littlewood-Polya (HLP) principle (see Section~\ref{sec:background})
then gives $\rho_{p,\phi}(T_Kf)\preceq\rho_{p,\phi}(T_Lf^*)$.  Thus
(1.1) follows from the one-dimensional inequality
\begin{align}
\left[\int_{\R_+}(T_Lf^*)^{**}(t)^q\phi_1^{(q)}(t)\,dt\right]^{1/q}
\preceq\left[\int_{\R_+}f^{**}(s)^p\phi_2(s)\,ds\right]^{1/p}<\infty,
\end{align}
where $\phi_1^{(q)}(t)=qt^{q-1}\int_t^\infty s^{-q}\phi_1(s)\,ds$.
The main result, Theorem~\ref{TH1N}, gives a further reduction of (1.2)
to a weighted norm inequality for the composed operator $T_LS$, where
$S$ denotes the Stieltjes transform
$(Sh)(t)=\int_{\R_+}h(s)/(t+s)\,ds$.  Its proof is in
Section~\ref{sec:proof}.

\section{Background}\label{sec:background}

A \emph{rearrangement-invariant} (r.i.) norm $\rho$ on $M_+(E)$,
$E\subset\R^n$, satisfies the axioms of \cite[Ch.\,1]{BS}, including
rearrangement-invariance: $\rho(f)=\rho(g)$ whenever $f^*=g^*$.  By
Luxemburg's theorem \cite[Ch.\,3, Thm.\,4.10]{BS} there is an r.i.\ norm
$\bar\rho$ on $M_+(\R_+)$ with $\rho(f)=\bar\rho(f^*)$ for all
$f\in M_+(E)$.  The K\"othe dual $\rho'(g)=\sup_{\rho(f)\le 1}\int_E fg$
satisfies $\rho''=\rho$ and H\"older's inequality
$\int_E fg\le\rho(f)\rho'(g)$.

\medskip
\noindent\textbf{Hardy-Littlewood-Polya (HLP) Principle}
\cite[Ch.\,2]{BS}.
\textit{For any r.i.\ norm $\rho$ on $M_+(E)$,}
\[
\int_0^t f^*\le\int_0^t g^*\;\text{ for all }t\in\R_+
\;\implies\;\rho(f)\le\rho(g).
\]

\medskip
The LG norm $\rho_{p,\phi}$ defined in \eqref{LGnorm}--\eqref{phi-cond}
is an r.i.\ norm; see \cite{GK}.  Its K\"othe dual satisfies
$\rho'_{p,\phi}\approx\rho_{p',\psi}$, where $p'=p/(p-1)$ and, for
$t\in\R_+$,
\begin{equation}\label{dualweight}
\psi(t)=\frac{t^{p'+p-1}
\bigl(\int_0^t\phi\bigr)\bigl(\int_t^\infty s^{-p}\phi(s)\,ds\bigr)}
{\bigl(\int_0^t\phi+t^p\int_t^\infty s^{-p}\phi(s)\,ds\bigr)^{p'+1}};
\end{equation}
see \cite{GolK,GP}.  In particular, $\phi\equiv 1$ gives
$\rho_{p,1}=\|\cdot\|_{L^p}$ and the dual weight $\psi\equiv 1$, while
$\phi(t)=t^{q/p-1}$ gives the Lorentz norm $\rho_{p,q}$ with dual weight
of the same power type.

The \emph{associated weight} $\phi^{(q)}$ defined by
\begin{equation}\label{phiq}
\phi^{(q)}(t)=qt^{q-1}\int_t^\infty s^{-q}\phi(s)\,ds, \quad t\in\R_+,
\end{equation}
arises naturally when passing from $f^*$ to $f^{**}$: it is the minimal
weight $v$ for which $\int_{\R_+} f^{**q}u\le\int_{\R_+} f^{*q}v$ holds
for all $f\in M_+(\R_+)$, as follows from differentiating the
characterization of Neugebauer~\cite{N} (see Proposition~\ref{prop:N}
below).

\section{Main Result}\label{sec:main}

\begin{theorem}\label{TH1N}
Fix the indices $p$ and $q$, $1<p\leq q< \infty$. Let $\phi_1, \phi_2
\in M_+(\R_+)$ be locally-integrable on $\R_+$. Set
\[
\phi_1^{(q)}(t)=qt^{q-1}\int_t^{\infty}s^{-q}\phi_1(s)\, ds,
\quad t\in \R_+.
\]
Suppose
\[
\int_{\R_+}\frac{\phi_1^{(q)}(t)}{1+t^q}\, dt, \quad
\int_{\R_+}\frac{\phi_2(t)}{1+t^p}\, dt < \infty
\]
and
\[
\int_{\R_+}\phi_1^{(q)}(t)\, dt=\int_{\R_+}\phi_2(t)\, dt=\infty.
\]
Denote the dual weights of $\phi_1^{(q)}$ and $\phi_2$ by $\psi_1^{(q)}$
and $\psi_2$, respectively, so that, for example,
\[
\psi_2(t)=\frac{t^{p+p'-1}\left(\int_0^t \phi_2\right)
\left(\int_t^{\infty}s^{-p}\phi_2(s)\, ds\right)}
{\left(\int_0^t \phi_2+t^p\int_t^{\infty}s^{-p}\phi_2(s)\,
ds\right)^{p'+1}}, \quad p'=p-1.
\]
Consider $K \in M_+(\R^{2n})$ and define
\[
L(t,s)=(K_x^*)^*_y(t,s), \quad s,t\in \R_+.
\]
Then (1.1) holds whenever (1.2) does. Moreover, the latter is the case
provided one has
\begin{align}
\left[\int_{\R_+}(T_LS)h(t)^q \psi_1^{(q)}(t)\,
dt\right]^{1/q}\leq
\left[\int_{\R_+}h(s)^p\psi_2(s)^{1-p}\, ds\right]^{1/p}<\infty,
\end{align}
$h\in M_+(\R_+)$. Here,
\begin{align*}
(T_LS)h(t)&=\int_{\R_+}L(t,s)\int_{\R_+}\frac{h(y)}{s+y}\, dyds\\
&=\int_{\R_+}h(y)\int_{\R_+}\frac{L(t,s)}{y+s}\, dsdy.
\end{align*}
\end{theorem}

\section{Proof of Theorem~\ref{TH1N}}\label{sec:proof}

The proof requires the following two results: the first is from \cite{N},
the second is from \cite{GolK}.

\begin{proposition}\label{prop:N}
Fix $q$, $1<q<\infty$, and suppose $u,v \in M_+(\R_+)$ are
locally-integrable on $\R_+$. Then
\begin{align}\label{3.1N}
\int_{\R_+}f^{**}(t)^qu(t)\, dt\leq
\int_{\R_+}f^*(t)^qv(t)\, dt< \infty,
\quad f\in M_+(\R_+),
\end{align}
if and only if
\begin{align}\label{3.2N}
\int_0^tu(s)\, ds+t^q\int_t^{\infty}s^{-q}u(s)\, ds\leq
\int_0^t v(s)\, ds, \quad t\in \R_+.
\end{align}
\end{proposition}

\begin{remark}
The smallest $v$ such that \eqref{3.1N} holds for a given $u$ satisfies,
by \eqref{3.2N},
\[
\int_0^tv(s)\, ds=\int_0^tu(s)\, ds+t^q\int_t^{\infty}s^{-q}u(s)\,
ds, \quad t\in \R_+.
\]
Differentiation yields
$v(t)=qt^{q-1}\int_t^{\infty}s^{-q}u(s)\, ds=u^{(q)}(t)$.
\end{remark}

\begin{proposition}\label{prop:GolK}
Let $p,q, \phi_1^{(q)}, \phi_2$ and $\psi_2$ be as in
Theorem~\ref{TH1N}. Then
\begin{align}\label{3.3N}
\left[\int_{\R_+}(T_Lf^{**})(t)^q\phi_1^{(q)}(t)\,
dt\right]^{1/q}\lesssim
\left[\int_{\R_+}f^{**}(s)^p\phi_2(s)\, ds\right]^{1/p}<\infty
\end{align}
if and only if
\begin{align}\label{3.4N}
\left[\int_{\R_+}(T_LSg)(t)^q\phi_1^{(q)}(t)\,
dt\right]^{1/q}\lesssim
\left[\int_{\R_+}g(s)^{p}\psi_2^{1-p}(s)\,
ds\right]^{1/p}< \infty,
\end{align}
$f\in\Gamma_{p, \phi_2}(\R_+)$ and $g\in L_p(\psi_2^{1-p})$.
\end{proposition}

\textbf{Proof of Theorem~\ref{TH1N}.}
We show that (1.1) holds whenever \eqref{3.3N} does. Indeed,
\begin{align*}
\left[\int_{\R_+}(T_Kf)^{**}(t)^q \phi_1(t)\,
dt\right]^{1/q}
&\leq \left[\int_{\R_+}(T_Lf^*)^{**}(t)^q\phi_1(t)\,
dt\right]^{1/q}\\
&\preceq
\left[\int_{\R_+}(T_Lf^{**})(t)^q\phi_1^{(q)}(t)\,
dt\right]^{1/q}\\
&\preceq
\left[\int_{\R_+}f^{**}(s)^p\phi_2(s)\,
ds\right]^{1/p}<\infty.
\end{align*}
According to Proposition~\ref{prop:GolK}, \eqref{3.3N} and \eqref{3.4N}
are equivalent. This concludes the proof. $\Box$

\section{Application: the case $L(t,s)=(t+s)^{-1}$}\label{sec:example}

\begin{remark}
Recall that
\[
(T_LSg)(t)=\int_{\R_+}g(s)\int_{\R_+}\frac{L(t,y)}{s+y}\,dyds,
\]
so the kernel of $T_LS$ is
\[
M(t,s)=\int_{\R_+}\frac{L(t,y)}{s+y}\,dy
\approx \frac{1}{s}\int_0^s L(t,y)\,dy+\int_s^{\infty}L(t,y)\frac{dy}{y}.
\]
Like $L$, $M$ is nonincreasing in each of $s$ and $t$.
\end{remark}

\begin{example}
Consider the case $L(t, s)=\dfrac{1}{t+s}$, so that $T_LS=S^2$. We have
\begin{align*}
(S^2f)(t)
&\approx \frac{1}{t}\int_0^t\left[\frac{1}{s}\int_0^s f(y)\,dy
  +\int_s^{\infty}f(y)\frac{dy}{y}\right]ds\\
&\quad+\int_t^{\infty}\left[\frac{1}{s}\int_0^s f(y)\,dy
  +\int_s^{\infty}f(y)\frac{dy}{y}\right]\frac{ds}{s}\\
&\approx \frac{1}{t}\int_0^t f(s)\log\frac{t}{s}\,ds
  +\int_t^{\infty}f(s)\log\frac{s}{t}\,\frac{ds}{s}.
\end{align*}
To obtain the condition on \eqref{3.4N} in this case we use the following
result from \cite{BK}.
\end{example}

\begin{theorem}[\cite{BK}]
Let $K(x,y)$ be nonnegative on $\R^2_+$, nondecreasing in $x$,
nonincreasing in $y$, satisfying the growth condition
\[
K(x,y)\leq D[K(x,z)+K(z,y)], \quad y<z<x.
\]
Let $t,u,v,w$ be nonnegative locally-integrable weight functions on
$\R_+$ and suppose $1<p\leq q<\infty$. Then
\begin{align}
\left(\int_{\R_+}\left(w(x)\int_0^x K(x,y)f(y)\,
dy\right)^qt(x)\,dx\right)^{1/q}
\leq
\left(\int_{\R_+}(u(y)f(y))^pv(y)\,dy\right)^{1/p},
\end{align}
$f\in M_+(\R_+)$, if and only if
\begin{align*}
&\left(\int_0^xK(x,y)u(y)^{-p'}v(y)^{1-p'}\,
dy\right)^{1/p'}\left(\int_x^{\infty}w(y)^{-q}t(y)^{1-q}\,
dy\right)^{1/q}\preceq1\\
&\textit{and}\\
&\left(\int_0^xu(y)^{-p'}v(y)^{1-p'}\,dy\right)^{1/p'}
\left(\int_x^{\infty}K(y,x)w(y)^{-q}t(y)^{1-q}\,
dy\right)^{1/q}\preceq 1.
\end{align*}
\end{theorem}

\begin{corollary}
Let $p,q, \phi_1^{(q)}$ and $\psi_2$ be as in Theorem~\ref{TH1N}. Then
\eqref{3.4N} holds for $T_LS=S^2$ if and only if
\begin{align*}
&\left(\int_0^x \ln \frac{x}{y}\,\psi_2(y)\,
dy\right)^{1/p'}\left(\int_x^{\infty}\phi_1^{(q)}(y)^{1-q}\,
dy\right)^{1/q}\preceq1,\\
&\left(\int_0^y \psi_2(x)\,
dx\right)^{1/p'}\left(\int_y^{\infty}\ln \frac{x}{y}\,
\phi_1^{(q)}(x)^{1-q}\,dx\right)^{1/q}\preceq 1,\\
&\left(\int_0^y \ln \frac{y}{x}\,\phi_1^{(q)}(x)^{1-q}\,
dx\right)^{1/q}\left(\int_y^{\infty}\psi_2(x)\,
dx\right)^{1/p'}\preceq 1,\\
&\left(\int_0^y \phi_1^{(q)}(x)^{1-q}\,
dx\right)^{1/q}\left(\int_y^{\infty}\ln
\frac{x}{y}\,\psi_2(x)\,dx\right)^{1/p'}\preceq 1.
\end{align*}
\end{corollary}


\end{document}